\documentclass[a4paper,12pt]{article}
    \usepackage[top=2.5cm,bottom=2.5cm,left=2.5cm,right=2.5cm]{geometry}
    \usepackage{cite, amsmath, amssymb}
    \usepackage[margin=1cm,%
                font=small,%
                format=hang,%
                labelsep=period,%
                labelfont=bf]{caption}
\usepackage[latin1]{inputenc}
\usepackage{rotating}
\usepackage{color}
\usepackage{graphicx}
\usepackage{epstopdf}
\usepackage{multirow} 
\usepackage{subfigure}
\usepackage{amsmath}
\usepackage{amsfonts}
\usepackage{latexsym}
\usepackage{authblk}
\usepackage[none]{hyphenat}
\begin{document}

\sloppy
\newcommand{\hide}[1]{}
\newcommand{\tbox}[1]{\mbox{\tiny #1}}
\newcommand{\half}{\mbox{\small $\frac{1}{2}$}}
\newcommand{\sinc}{\mbox{sinc}}
\newcommand{\const}{\mbox{const}}
\newcommand{\trc}{\mbox{trace}}
\newcommand{\intt}{\int\!\!\!\!\int }
\newcommand{\ointt}{\int\!\!\!\!\int\!\!\!\!\!\circ\ }
\newcommand{\eexp}{\mbox{e}^}
\newcommand{\bra}{\left\langle}
\newcommand{\ket}{\right\rangle}
\newcommand{\EPS} {\mbox{\LARGE $\epsilon$}}
\newcommand{\ar}{\mathsf r}
\newcommand{\im}{\mbox{Im}}
\newcommand{\re}{\mbox{Re}}
\newcommand{\bmsf}[1]{\bm{\mathsf{#1}}}
\newcommand{\mpg}[2][1.0\hsize]{\begin{minipage}[b]{#1}{#2}\end{minipage}}
\newcommand{\CC}{\mathbb{C}}
\newcommand{\NN}{\mathbb{N}}
\newcommand{\PP}{\mathbb{P}}
\newcommand{\RR}{\mathbb{R}}
\newcommand{\QQ}{\mathbb{Q}}
\newcommand{\ZZ}{\mathbb{Z}}
\renewcommand{\a}{\alpha}
\renewcommand{\b}{\beta}
\renewcommand{\d}{\delta}
\newcommand{\D}{\Delta}
\newcommand{\g}{\gamma}
\newcommand{\G}{\Gamma}
\renewcommand{\th}{\theta}
\renewcommand{\l}{\lambda}
\renewcommand{\L}{\Lambda}
\renewcommand{\O}{\Omega}
\newcommand{\s}{\sigma}

\newtheorem{theorem}{Theorem}
\newtheorem{acknowledgement}[theorem]{Acknowledgement}
\newtheorem{algorithm}[theorem]{Algorithm}
\newtheorem{axiom}[theorem]{Axiom}
\newtheorem{claim}[theorem]{Claim}
\newtheorem{conclusion}[theorem]{Conclusion}
\newtheorem{condition}[theorem]{Condition}
\newtheorem{conjecture}[theorem]{Conjecture}
\newtheorem{corollary}[theorem]{Corollary}
\newtheorem{criterion}[theorem]{Criterion}
\newtheorem{definition}[theorem]{Definition}
\newtheorem{example}[theorem]{Example}
\newtheorem{exercise}[theorem]{Exercise}
\newtheorem{lemma}[theorem]{Lemma}
\newtheorem{notation}[theorem]{Notation}
\newtheorem{problem}[theorem]{Problem}
\newtheorem{proposition}[theorem]{Proposition}
\newtheorem{remark}[theorem]{Remark}
\newtheorem{solution}[theorem]{Solution}
\newtheorem{summary}[theorem]{Summary}
\newenvironment{proof}[1][Proof]{\noindent\textbf{#1.} }{\ \rule{0.5em}{0.5em}}

\title{\vspace*{3.5cm}\textbf{Analytical and computational properties of the variable symmetric division deg index}}

\author[1]{\bf{R. Aguilar-S\'anchez}}
\author[2]{\bf{J. A. M\'endez-Berm\'udez}}
\author[3]{\bf{Jos\'e M. Rodr\'{\i}guez}}
\author[4]{\bf{Jos\'e M. Sigarreta\footnote{Corresponding author}}}

\affil[1]{\small{\it Facultad de Ciencias Qu\'imicas, Benem\'erita Universidad
Aut\'onoma de Puebla, Puebla 72570, Mexico}}
\affil[2]{\it Instituto de F\'{\i}sica, Benem\'erita Universidad Aut\'onoma de Puebla,
Apartado Postal J-48, Puebla 72570, Mexico}
\affil[3]{\it Departamento de Matem\'aticas, Universidad Carlos III de Madrid, 
Avenida de la Universidad 30, 28911 Legan\'es, Madrid, Spain}
\affil[4]{\it Facultad de Matem\'aticas, Universidad Aut\'onoma de Guerrero, 
Carlos E. Adame No.54 Col. Garita, Acapulco Gro. 39650, Mexico}
\affil[ ]{\ttfamily {\textbf {ras747698@gmail.com, jmendezb@ifuap.buap.mx, jomaro@math.uc3m.es, jsmathguerrero@gmail.com}}}
\date{}

\maketitle
\thispagestyle{empty}

\centerline{(Received xxx)}
\begin{abstract}
The aim of this work is to obtain new inequalities for the variable symmetric division deg index 
$SDD_\alpha(G) = \sum_{uv \in E(G)} (d_u^\alpha/d_v^\alpha+d_v^\alpha/d_u^\alpha)$, 
and to characterize graphs extremal with respect to them.
Here, $uv$ denotes the edge of the graph $G$ connecting the vertices $u$ and $v$,
$d_u$ is the degree of the vertex $u$, and $\a \in \RR$.
Some of these inequalities generalize and improve previous results for the symmetric division 
deg index.
In addition, we computationally apply the $SDD_\alpha(G)$ index on 
random graphs and show that the ratio $\bra SDD_a(G) \ket/n$ ($n$ being the order of the graph) 
depends only on the average degree $\bra d \ket$.
\end{abstract}

\baselineskip=0.30in

\section{Preliminaries}

A large number of graph invariants of the form $X(G)=\sum_{uv \in E(G)} F(d_u,d_v)$, where $uv$ 
denotes the edge of the graph $G$ connecting the vertices $u$ and $v$ and $d_u$ is the degree of 
the vertex $u$, are studied in mathematical chemistry.
The single number $X(G)$, representing a chemical structure in graph-theoretical terms via the molecular 
graph, is called a topological descriptor and if in addition it correlates with a molecular property it is 
called a topological index.
Remarkably, topological indices capture physical properties of a molecule in a single number.

Hundreds of topological indices have been proposed and studied over more than 40 years.
Among them, probably the most popular descriptors are the Randi\'c connectivity index and the 
Zagreb indices. The first and second Zagreb indices,
denoted by $M_1$ and $M_2$, respectively, were introduced by Gutman and Trinajsti\'c in $1972$ (see \cite{GT})
as
$$
M_1(G) = \sum_{u\in V(G)} d_u^2,
\qquad
M_2(G) = \sum_{uv\in E(G)} d_u d_v .
\qquad
$$
For details of the applications and mathematical theory of  Zagreb indices see \cite{Gutman,GD,GR},
and the references therein.

The concept of variable molecular descriptors was proposed as a new way of characterizing heteroatoms 
in molecules (see \cite{R2,R3}), but also to assess structural differences (e.g., the relative role of carbon 
atoms of acyclic and cyclic parts in alkylcycloalkanes \cite{RPL}).
The idea behind the variable molecular descriptors is that the variables are determined during the
regression so that the standard error of estimate for a particular studied property is as small as possible 
(see, e.g., \cite{MN}).

In this line of ideas, the variable versions of the first and second Zagreb indices were defined as 
\cite{LZheng,LZhao,MN} 
$$
M_1^{\a}(G) = \sum_{u\in V(G)} d_u^{\a},
\qquad
M_2^{\a}(G) = \sum_{uv\in E(G)} (d_u d_v)^\a ,
$$
with $\a \in \RR$.
Evidently, $M_1^{2}$ and $M_2^{1}$ are the first and second Zagreb indices, respectively. In addition,
\emph{the first and second variable Zagreb indices} include several known indices. As examples
we note that $M_1^{-1}$ is the inverse index $ID$, $M_1^{3}$ is the forgotten index $F$, 
$M_2^{-1/2}$ is the Randi\'c index, and $M_2^{-1}$ is the modified Zagreb index.

In 2011, Vuki\v{c}evi\'c proposed the \emph{variable symmetric division deg index} \cite{Vuki5}
\begin{equation}
SDD_\alpha(G)
= \sum_{uv\in E(G)} \left(\frac{d_u^\alpha}{d_v^\alpha} + \frac{d_v^\alpha}{d_u^\alpha} \right) .
\label{SDD}
\end{equation}
Note that $SDD_{-\alpha}(G)=SDD_\alpha(G)$ and so, it suffices to consider positive values of $\a$.
The symmetric division deg index is the best predictor of total surface area for polychlorobiphenyls \cite{VG}.

\bigskip

In this paper we perform studies of the variable symmetric division deg index from analytical and 
computational viewpoints.

\section{Analytical study of the variable symmetric division deg index}
\label{analytics}

Let us start by proving a monotonicity property of these indices.

\begin{theorem} \label{t:1}
Let $G$ be a graph and $0 < \a < \b$.
Then
$$
SDD_{\a}(G)
\le SDD_{\b}(G) ,
$$
and the equality in the bound is attained if and only if each connected component of $G$ is a regular graph.
\end{theorem}

\begin{proof}
Let us consider $x \ge 1$.
Thus, $x^{\a} \ge x^{-\b}$ and
$$
\begin{aligned}
x^{\b-\a} - 1 \ge 0,
\qquad
& x^{\a} \big(x^{\b-\a} - 1\big)
\ge x^{-\b} \big(x^{\b-\a} - 1\big),
\\
x^{\b} - x^{\a}
\ge x^{-\a} - x^{-\b} ,
\qquad
& x^{\b} + x^{-\b}
\ge x^{\a} + x^{-\a},
\end{aligned}
$$
for every $x \ge 1$.
Since $u(x)=x^{\a} + x^{-\a}$ satisfies $u(1/x)=u(x)$ for every $x>0$, we have
$x^{\b} + x^{-\b} \ge x^{\a} + x^{-\a}$ for every $x>0$.
Note that the equality is attained if and only if $x=1$.

Thus, we have
$$
\begin{aligned}
SDD_{\b}(G)
& = \sum_{uv \in E(G)} \!\! \left( \,\frac{d_u^{\b}}{d_v^{\b}} + \frac{d_v^{\b}}{d_u^{\b}} \right)
\ge \sum_{uv \in E(G)} \!\! \left( \,\frac{d_u^{\a}}{d_v^{\a}} + \frac{d_v^{\a}}{d_u^{\a}} \right)
= SDD_{\a}(G).
\end{aligned}
$$

The previous argument gives that the equality in the bound is attained if and only if $d_u/d_v = 1$ for every $uv \in E(G)$,
i.e., each connected component of $G$ is a regular graph.
\end{proof}

The next result relates the $SDD_a$ and the variable Zagreb indices.

\begin{theorem} \label{t:m2m1}
If $G$ is a graph with minimum degree $\d$ and maximum degree $\D$, and $\a >0$, then
$$
\begin{aligned}
2\d^{2\a} M_2^{-\a}(G)
& \le SDD_{\a}(G)
\le 2\D^{2\a} M_2^{-\a}(G),
\\
\D^{-2\a} M_1^{2\a+1}(G)
& \le SDD_{\a}(G)
\le \d^{-2\a} M_1^{2\a+1}(G),
\end{aligned}
$$
and the equality in each bound is attained if and only if $G$ is regular.
\end{theorem}

\begin{proof}
First of all recall that for any function $f$ we have
$$
\sum_{uv \in E(G)} \big( f(d_u)+ f(d_v) \big)
= \sum_{u \in V(G)} d_u f(d_u) .
$$
In particular,
$$
\sum_{uv \in E(G)} \big( d_u^{2\a}+d_v^{2\a} \big)
= \sum_{u \in V(G)} d_u^{2\a+1}
= M_1^{a+1}(G).
$$
Since
$$
SDD_{\a}(G)
= \sum_{uv \in E(G)} \!\! \left( \,\frac{d_u^{\a}}{d_v^{\a}} + \frac{d_v^{\a}}{d_u^{\a}} \right)
= \sum_{uv \in E(G)} \frac{d_u^{2\a}+d_v^{2\a}}{(d_u d_v)^{\a}} \,.
$$
and $\a > 0$, we obtain
$$
SDD_{\a}(G)
= \sum_{uv \in E(G)} \frac{d_u^{2\a}+d_v^{2\a}}{(d_u d_v)^{\a}}
\le 2\D^{2\a} \!\! \!\! \sum_{uv \in E(G)} (d_u d_v)^{-\a}
= 2\D^{2\a} M_2^{-\a}(G) ,
$$
and
$$
SDD_{\a}(G)
= \sum_{uv \in E(G)} \frac{d_u^{2\a}+d_v^{2\a}}{(d_u d_v)^{\a}}
\ge 2\d^{2\a} \!\! \!\! \sum_{uv \in E(G)} (d_u d_v)^{-\a}
= 2\d^{2\a} M_2^{-\a}(G) .
$$

We also have
$$
\begin{aligned}
SDD_{\a}(G)
& = \sum_{uv \in E(G)} \frac{d_u^{2\a}+d_v^{2\a}}{(d_u d_v)^{\a}}
\le \d^{-2\a} \!\! \!\! \sum_{uv \in E(G)} \big(d_u^{2\a}+d_v^{2\a}\big)
\\
& = \d^{-2\a} \!\! \!\! \sum_{u \in V(G)} d_u^{2\a+1}
= \d^{-2\a} M_1^{2\a+1}(G) ,
\end{aligned}
$$
and
$$
\begin{aligned}
SDD_{\a}(G)
& = \sum_{uv \in E(G)} \frac{d_u^{2\a}+d_v^{2\a}}{(d_u d_v)^{\a}}
\ge \D^{-2\a} \!\! \!\! \sum_{uv \in E(G)} \big(d_u^{2\a}+d_v^{2\a}\big)
\\
& = \D^{-2\a} \!\! \!\! \sum_{u \in V(G)} d_u^{2\a+1}
= \D^{-2\a} M_1^{2\a+1}(G) .
\end{aligned}
$$

\smallskip

If $G$ is a regular graph, then each lower bound and its corresponding upper bound are the same, and they are equal to $SDD_{\a}(G)$.

Assume now that the equality in either the first or second bound is attained.
The previous argument gives that we have either $d_u^{2\a}+d_v^{2\a} = 2\D^{2\a}$ for every $uv \in E(G)$
or $d_u^{2\a}+d_v^{2\a} = 2\d^{2\a}$ for every $uv \in E(G)$.
Thus, $d_u=d_v=\D$ for every $uv\in E(G)$, or $d_u=d_v=\d$ for every $uv\in E(G)$.
Hence, $G$ is regular.

Finally, assume that the equality in either the third or fourth bound is attained.
The previous argument gives that we have either $(d_u d_v)^{\a} = \d^{2\a}$ for every $uv \in E(G)$
or $(d_u d_v)^{\a} = \D^{2\a}$ for every $uv \in E(G)$.
Thus, $d_u=d_v=\d$ for every $uv\in E(G)$, or $d_u=d_v=\D$ for every $uv\in E(G)$.
Therefore, $G$ is regular.
\end{proof}

We will need the following technical result.

\begin{lemma} \label{l:1}
Let $0<a<A$. Then
$$
a
\le \frac{x^{2}+y^{2}}{x + y}
\le A
$$
for every $a \le x,y \le A$.
The lower bound is attained if and only if $x=y=a$.
The upper bound is attained if and only if $x=y=A$.
\end{lemma}

\begin{proof}
If $a \le x,y \le A$, then $ax+ay\le x^{2}+y^{2} \le Ax+Ay$, and the statement holds.
\end{proof}

A family of degree--based topological indices, named \emph{Adriatic indices}, was put
forward in \cite{VG,V2}. Twenty of them were selected as significant predictors.
One of them, the \emph{inverse sum indeg} index, $ISI$, was singled out
in \cite{VG,V2} as a significant predictor of total surface area of octane isomers.
This index is defined as
$$
ISI(G) = \sum_{uv\in E(G)} \frac{d_u\,d_v}{d_u + d_v}
= \sum_{uv\in E(G)} \frac{1}{\frac{1}{d_u} + \frac{1}{d_v}}\,.
$$
Next, we relate $SDD_{\a}(G)$ with the \emph{variable inverse sum deg index} defined, for each $a \in \RR$, as
$$
ISD_a(G)
= \sum_{uv \in E(G)} \frac{1}{d_u^a + d_v^a} \,.
$$
Note that $ISD_{-1}$ is the inverse sum indeg index $ISI$.

\begin{theorem} \label{t:isd}
If $G$ is a graph with $m$ edges and minimum degree $\d$, and $\a>0$, then
$$
SDD_{\a}(G)
\ge \frac{\d^{\a} m^2}{ISD_{-\a}(G)} \,,
$$
and the equality in the bound is attained if and only if $G$ is regular.
\end{theorem}

\begin{proof}
Cauchy-Schwarz inequality gives
$$
\begin{aligned}
m^2
& = \Big( \sum_{uv \in E(G)} 1 \Big)^2
= \Big( \sum_{uv \in E(G)} \Big( \frac{d_u^{2\a}+d_v^{2\a}}{d_u^{\a} d_v^{\a}}\Big)^{1/2}
\Big( \frac{d_u^{\a} d_v^{\a}}{d_u^{2\a}+d_v^{2\a}}\Big)^{1/2} \, \Big)^2
\\
& \le \sum_{uv \in E(G)} \frac{d_u^{2\a}+d_v^{2\a}}{d_u^{\a} d_v^{\a}}
\sum_{uv \in E(G)} \frac{d_u^{\a} d_v^{\a}}{d_u^{2\a}+d_v^{2\a}} \,.
\end{aligned}
$$
Since Lemma \ref{l:1} gives
$$
\d^{\a}
\le \frac{x^{2\a}+y^{2\a}}{x^{\a} + y^{\a}}
\le \D^{\a},
\qquad
\frac1{x^{2\a}+y^{2\a}}
\le \frac{\d^{-\a}}{x^{\a} + y^{\a}}
\,,
$$
for every $\d \le x,y \le \D$,
we have
$$
\begin{aligned}
m^2
& \le \sum_{uv \in E(G)} \frac{d_u^{2\a}+d_v^{2\a}}{d_u^{\a} d_v^{\a}}
\sum_{uv \in E(G)} \frac{d_u^{\a} d_v^{\a}}{d_u^{2\a}+d_v^{2\a}}
\\
& \le \d^{-\a} SDD_{\a}(G)
\sum_{uv \in E(G)} \frac{d_u^{\a} d_v^{\a}}{d_u^{\a}+d_v^{\a}}
\\
& = \d^{-\a} SDD_{\a}(G)
\sum_{uv \in E(G)} \frac{1}{d_u^{-\a}+d_v^{-\a}}
\\
& = \d^{-\a} SDD_{\a}(G)\, ISD_{-\a}(G).
\end{aligned}
$$

If $G$ is a regular graph, then
$SDD_{\a}(G)=2m$,
$ISD_{-\a}(G)= m \d^{\a}/2$
and the equality in the bound is attained.

Assume now that the equality in the bound is attained.
Thus, by the previous argument and Lemma \ref{l:1}
we have $d_u=d_v=\d$ for every $uv\in E(G)$.
Hence, $G$ is regular.
\end{proof}

\medskip

The \emph{modified Narumi-Katayama index}
$$
NK^*(G) = \prod_{u\in V (G)} d_u^{d_u} = \prod_{uv\in E (G)} d_u d_v
$$
is introduced in \cite{GSG}, inspired in the Narumi-Katayama index defined in \cite{NK}.
Next, we prove an inequality relating the modified Narumi-Katayama index with $SDD_{\a}(G)$.

\begin{theorem} \label{t:nk}
Let $G$ be a graph with $m$ edges and minimum degree $\d$, and $\a > 0$.
Then
$$
SDD_\a(G) \ge 2\d^{2\a} m NK^*(G)^{-\a/m},
$$
and the equality in the bound is attained if and only if $G$ is a regular graph.
\end{theorem}

\begin{proof}
Using the fact that the geometric mean is at most the arithmetic mean, we obtain
$$
\begin{aligned}
\frac{1}{m}\,SDD_\a(G)
& = \frac{1}{m} \sum_{uv \in E(G)} \!\! \left( \,\frac{d_u^{\a}}{d_v^{\a}} + \frac{d_v^{\a}}{d_u^{\a}} \right)
= \frac{1}{m} \sum_{uv \in E(G)} \frac{d_u^{2\a}+d_v^{2\a}}{(d_u d_v)^{\a}}
\\
& \ge 2\d^{2\a} \frac{1}{m} \sum_{uv\in E(G)} \frac1{(d_u d_v)^\a}
\ge 2\d^{2\a} \left(\, \prod_{uv\in E(G)} \frac1{(d_u d_v)^\a} \,\right)^{1/m}
\\
& = 2\d^{2\a} NK^*(G)^{-\a/m}.
\end{aligned}
$$

If $G$ is a regular graph, then
$$
2\d^{2\a} m NK^*(G)^{-\a/m}
= 2\d^{2\a} m \big( \d^{2 m} \big)^{-\a/m}
= 2m
= SDD_\a(G) .
$$

Finally, assume that the equality in the bound is attained.
The previous argument gives that
$d_u^{2\a}+d_v^{2\a}= 2\d^{2\a}$
for every $uv \in E(G)$, and so,
$d_u=d_v=\d$
for every $uv \in E(G)$.
Hence, $G$ is a regular graph.
\end{proof}

\smallskip

Next, we obtain additional bounds of $SDD_{\a}$ which do not involve other topological indices.

\begin{theorem} \label{t:alfa}
Let $G$ be a graph with $m$ edges, minimum degree $\d$ and maximum degree $\d+1$, $\a>0$
and $A$ the cardinality of the set of edges $uv \in E(G)$ with $d_u \ne d_v$.
Then $A$ is an even integer and
$$
SDD_{\a}(G)
= 2m + A \left(  \frac{(\d+1)^{\a}}{\d^{\a}} + \frac{\d^{\a}}{(\d+1)^{\a}} - 2 \right)  .
$$
\end{theorem}

\begin{proof}
Let $F=\left\lbrace uv\in E(G):\; d_u \ne d_v \right\rbrace $, then $A$ is the cardinality of the set $F$.
Since the minimum degree of $G$ is $\d$ and its maximum degree is $\d+1$,
if $uv\in F$, then $d_u=\d$ and $d_v=\d+1$ or viceversa, and therefore
$$
\frac{d_u^{\a}}{d_v^{\a}} + \frac{d_v^{\a}}{d_u^{\a}}
=\frac{(\d+1)^{\a}}{\d^{\a}} + \frac{\d^{\a}}{(\d+1)^{\a}} \,.
$$	
If $uv\in F^c=E(G) \setminus F$, then $d_u=d_v=\d$ or $d_u=d_v=\d+1$, and therefore
$$
\frac{d_u^{\a}}{d_v^{\a}} + \frac{d_v^{\a}}{d_u^{\a}}=2 .
$$
Since there are exactly $A$ edges in $F$ and $m-A$ edges in $F^c$, we have
$$
\begin{aligned}
SDD_{\a}(G)&=\!\!\!\!\!\!\sum_{uv \in E(G)} \!\! \Big( \,\frac{d_u^{\a}}{d_v^{\a}} + \frac{d_v^{\a}}{d_u^{\a}} \Big)\\
&=\!\!\!\!\sum_{uv \in F^c} \! \Big( \,\frac{d_u^{\a}}{d_v^{\a}} + \frac{d_v^{\a}}{d_u^{\a}} \Big)+\!\!\sum_{uv \in F} \! \Big( \,\frac{d_u^{\a}}{d_v^{\a}} + \frac{d_v^{\a}}{d_u^{\a}} \Big)\\
&=\!\!\!\!\sum_{uv \in F^c} \!2 +\!\!\sum_{uv \in F} \! \Big( \, \frac{(\d+1)^{\a}}{\d^{\a}} + \frac{\d^{\a}}{(\d+1)^{\a}} \Big) \\
&=2m-2A + A\Big( \, \frac{(\d+1)^{\a}}{\d^{\a}} + \frac{\d^{\a}}{(\d+1)^{\a}} \Big) .
\end{aligned}
$$
This gives the equality.

Seeking for a contradiction assume that $A$ is an odd integer.

Let $\G_1$ be the subgraph of $G$ induced by the $n_1$ vertices with degree $\d$ in $V(G)$,
and denote by $m_1$ the cardinality of the set of edges of $\G_1$.
Handshaking Lemma gives $n_1\d - A = 2m_1$.
Since $A$ is an odd integer, $\d$ is also an odd integer.
Thus, $\d+1$ is an even integer.

Let $\G_2$ be the subgraph of $G$ induced by the $n_2$ vertices with degree $\d+1$ in $V(G)$,
and denote by $m_2$ the cardinality of the set of edges of $\G_2$.
Handshaking Lemma gives $n_2(\d+1) - A = 2m_2$, a contradiction,
since $A$ is an odd integer and $\d+1$ is an even integer.

Thus, we conclude that $A$ is an even integer.
\end{proof}

We will need the following result in the proof of Theorem \ref{t:dD} below.

\begin{lemma} \label{l:u}
Given $\a>0$, consider the function $u: (0,\infty) \rightarrow (0,\infty)$ defined as
$u(t)=t^{\a}+t^{-\a}$.
Then $u$ is strictly decreasing on $(0,1]$, $u$ is strictly increasing on $[1,\infty)$
and $u(t) \ge u(1)=2$.
\end{lemma}

\begin{proof}
We have
$$
u'(t)
=\a t^{\a-1}-\a t^{-\a-1}
=\a t^{-\a-1}( t^{2\a}-1).
$$
Since $\a>0$, we have $u'<0$ on $(0,1)$ and $u'>0$ on $[1,\infty)$.
This gives the result.
\end{proof}

Theorem \ref{t:alfa} gives the precise value of $SDD_{\a}$ when $\D=\d+1$.
Theorem \ref{t:dD} below provides a lower bound when $\D>\d+1$.

\begin{theorem} \label{t:dD}
	Let $G$ be a graph with $m$ edges, minimum degree $\d$ and maximum degree $\D>\d+1$.
	Denote by $A_0,A_1,A_2,$
	the cardinality of the subsets of edges
	$F_0 = \{ uv \in E(G) : \, d_{u} = \d, d_{v} = \D \}$,
	$F_1 = \{ uv \in E(G) : \, d_{u} = \d, \d < d_{v} < \D \}$,
	$F_2 = \{ uv \in E(G) : \, d_{u} = \D, \d < d_{v} < \D \}$,
	respectively.
	If $\a>0$, then
	$$
\begin{aligned}
	SDD_{\a}(G)
	& \le
	(m-A_1-A_2) \left(\,\frac{\D^{\a}}{\d^{\a}} + \frac{\d^{\a}}{\D^{\a}}\right)
+A_1 \! \left( \frac{(\D-1)^{\a}}{\d^{\a}} + \frac{\d^{\a}}{(\D-1)^{\a}}\right)
\\
& \quad +A_2 \!\left( \frac{\D^{\a}}{(\d+1)^{\a}} + \frac{(\d+1)^{\a}}{\D^{\a}}\right) ,
	\\
	SDD_{\a}(G) & \ge
	2m  + A_0 \! \left( \frac{\D^{\a}}{\d^{\a}} + \frac{\d^{\a}}{\D^{\a}}-2\right)
+ A_1 \! \left( \frac{(\d+1)^{\a}}{\d^{\a}} + \frac{\d^{\a}}{(\d+1)^{\a}}-2\right)
\\
& \quad + A_2 \! \left(  \frac{\D^{\a}}{(\D-1)^{\a}} + \frac{(\D-1)^{\a}}{\D^{\a}}-2\right) .
\end{aligned}
	$$
\end{theorem}

\begin{proof}
	Lemma \ref{l:u} gives that the function
	$$
	\frac{d_{v}^{\a}}{\d^{\a}} + \frac{\d^{\a}}{d_{v}^{\a}}
	= u \left(  \frac{d_{v}}{\d} \right)
	$$
	is increasing in $d_{v} \in [\d+1,\D-1]$ and so,
	$$
	\frac{(\d+1)^{\a}}{\d^{\a}} + \frac{\d^{\a}}{(\d+1)^{\a}}
	\le
	\frac{d_{v}^{\a}}{\d^{\a}} + \frac{\d^{\a}}{d_{v}^{\a}}
	\le
	\frac{(\D-1)^{\a}}{\d^{\a}} + \frac{\d^{\a}}{(\D-1)^{\a}}
	\,,
	$$
	for every $uv \in F_1$.
	
	In a similar way, Lemma \ref{l:u} gives that the function
	$$
	\frac{\D^{\a}}{d_{v}^{\a}} + \frac{d_{v}^{\a}}{\D^{\a}}
	= u \left(  \frac{d_{v}}{\D} \right)
	$$
	is decreasing in $d_{v} \in [\d+1,\D-1]$ and so,
	$$
	\frac{\D^{\a}}{(\D-1)^{\a}} + \frac{(\D-1)^{\a}}{\D^{\a}}
	\le
	\frac{\D^{\a}}{d_{v}^{\a}} + \frac{d_{v}^{\a}}{\D^{\a}}
	\le
	\frac{\D^{\a}}{(\d+1)^{\a}} + \frac{(\d+1)^{\a}}{\D^{\a}}
	\,,
	$$
	for every $uv \in F_2$.
	
	Also,
	$$
	2
\le \frac{d_{u}^{\a}}{d_{v}^{\a}} + \frac{d_{v}^{\a}}{d_{u}^{\a}}
\le \frac{\D^{\a}}{\d^{\a}} + \frac{\d^{\a}}{\D^{\a}}
	$$
	for every $uv \in E(G)$.

We have
	$$
	\begin{aligned}
	SDD_{\a}(G)
	& = \sum_{uv\in E(G)\setminus (F_0 \cup F_1 \cup F_2)} \left(\, \frac{d_{u}^{\a}}{d_{v}^{\a}} + \frac{d_{v}^{\a}}{d_{u}^{\a}} \,\right)
	+ \sum_{uv\in F_0} \left(\, \frac{d_{u}^{\a}}{d_{v}^{\a}} + \frac{d_{v}^{\a}}{d_{u}^{\a}} \,\right)
	\\
& \quad + \sum_{uv\in F_1} \left(\, \frac{d_{u}^{\a}}{d_{v}^{\a}} + \frac{d_{v}^{\a}}{d_{u}^{\a}} \,\right)
	+ \sum_{uv\in F_2} \left(\, \frac{d_{u}^{\a}}{d_{v}^{\a}} + \frac{d_{v}^{\a}}{d_{u}^{\a}} \,\right)
	\\
	& \ge \sum_{uv\in E(G)\setminus (F_0 \cup F_1 \cup F_2)} \!\!\!\!\!\!\!\!\!\!\!\! 2 \,\,\,\,\,\,
	+ \sum_{uv\in F_0} \left(\, \frac{\D^{\a}}{\d^{\a}} + \frac{\d^{\a}}{\D^{\a}} \,\right)
	\\
& \quad + \sum_{uv\in F_1} \left(\, \frac{d_v^{\a}}{\d^{\a}} + \frac{\d^{\a}}{d_v^{\a}} \,\right)
	+ \sum_{uv\in F_2} \left(\, \frac{\D^{\a}}{d_v^{\a}} + \frac{d_v^{\a}}{\D^{\a}} \,\right) .
	\end{aligned}
	$$
Hence,
	$$
\begin{aligned}
	SDD_{\a}(G)
& \ge
	2m - 2A_0 - 2A_1 - 2A_2 + A_0 \left(\, \frac{\D^{\a}}{\d^{\a}} + \frac{\d^{\a}}{\D^{\a}} \,\right)
\\
& \quad + A_1 \left(\, \frac{(\d+1)^{\a}}{\d^{\a}} + \frac{\d^{\a}}{(\d+1)^{\a}} \,\right)
+ A_2 \left(\, \frac{\D^{\a}}{(\D-1)^{\a}} + \frac{(\D-1)^{\a}}{\D^{\a}} \,\right)  .
\end{aligned}
$$

We also have
$$
\begin{aligned}
	SDD_{\a}(G)
	& = \sum_{uv\in E(G)\setminus (F_1 \cup F_2)} \left(\, \frac{d_{u}^{\a}}{d_{v}^{\a}} + \frac{d_{v}^{\a}}{d_{u}^{\a}} \,\right)
	\\
& \quad + \sum_{uv\in F_1} \left(\, \frac{d_{u}^{\a}}{d_{v}^{\a}} + \frac{d_{v}^{\a}}{d_{u}^{\a}} \,\right)
	+ \sum_{uv\in F_2} \left(\, \frac{d_{u}^{\a}}{d_{v}^{\a}} + \frac{d_{v}^{\a}}{d_{u}^{\a}} \,\right)
	\\
	& \le
	(m - A_1 - A_2)\left(\, \frac{\D^{\a}}{\d^{\a}} + \frac{\d^{\a}}{\D^{\a}} \,\right)
\\
& \quad + A_1 \left(\, \frac{(\D-1)^{\a}}{\d^{\a}} + \frac{\d^{\a}}{(\D-1)^{\a}} \,\right)
+ A_2 \left(\, \frac{\D^{\a}}{(\d+1)^{\a}} + \frac{(\d+1)^{\a}}{\D^{\a}} \,\right) .
\end{aligned}
$$
\end{proof}

\section{Computational study of the variable symmetric division deg index on random graphs}
\label{statistics}

Here we consider two models of random graphs $G$: Erd\"os-R\'enyi (ER) graphs
$G(n,p)$ and bipartite random (BR) graphs $G(n_1,n_2,p)$. 
ER graphs are formed by $n$ vertices connected independently with probability $p \in [0,1]$. 
While BR graphs are composed by two disjoint sets, set 1 and set 2, with $n_1$ and $n_2$ vertices each such 
that there are no adjacent vertices within the same set, being $n=n_1+n_2$ the total number 
of vertices in the bipartite graph. The vertices of the two sets are connected randomly with
probability $p \in [0,1]$.

We stress that the computational study of the variable symmetric division deg index we perform 
below is justified by the random nature of the graph models we want to explore. Since a given 
parameter set $(n,p)$ [$(n_1,n_2,p)$] represents an infinite-size ensemble of ER graphs 
[BR graphs], the computation of $SDD_\alpha(G)$ on a single graph is irrelevant. In contrast, 
the computation of $\left<  SDD_\alpha(G) \right>$ (where $\left<  \cdot \right>$ indicates 
ensemble average) over a large number of random graphs, all characterized by the same 
parameter set $(n,p)$ [$(n_1,n_2,p)$], may provide useful {\it average} 
information about the full ensemble. This {\it computational} approach, well known in random matrix 
theory studies, is not widespread in studies involving topological indices, mainly because 
topological indices are not commonly applied to random graphs; for very recent exceptions 
see~\cite{MMRS20,MMRS21,AMRS21,AHMG20}.

\subsection{Average properties of the $SDD_\alpha$ index on Erd\"os-R\'enyi random graphs}
\label{ER}

From the definition of the variable symmetric division deg index, see Eq.~(\ref{SDD}), we have that:
\begin{itemize}
\item[{\bf (i)}]
For $\alpha=0$, $\left< SDD_0(G)\right>$ gives twice the average number of edges of the ER
graph. That is,
\begin{eqnarray}
\label{SDD0}
\left<  SDD_0(G) \right> & = & \left<\sum_{uv\in E(G)} \left(\frac{d_u^0}{d_v^0} + \frac{d_v^0}{d_u^0} \right)\right>
= \left< \sum_{uv\in E(G)} (1+1) \right> \nonumber \\
& = & \left< 2 |E(G)| \right> = n(n-1) p \, .
\end{eqnarray}
\item[{\bf (ii)}]
When $np\gg 1$, we can approximate $d_u \approx d_v \approx \left< d \right>$, then 
\begin{eqnarray}
\left<  SDD_\alpha(G) \right> & \approx &
\left<\sum_{uv\in E(G)} \left( 1^\alpha + 1^\alpha \right)\right> = \left< \sum_{uv\in E(G)} 2 \right> \nonumber \\
& = & \left< 2 |E(G)| \right> = n(n-1) p \, .
\label{avSDD}
\end{eqnarray}
\item[{\bf (iii)}]
By recognizing that the average degree of the ER graph model reads as
\begin{equation}
\label{k}
\left< d \right> = (n-1)p \, ,
\end{equation}
we can rewrite Eq.~(\ref{avSDD}) as
\begin{equation}
\frac{\left<  SDD_\alpha(G) \right>}{n} \approx \left< d \right> \, .
\label{avSDD2}
\end{equation}
We stress that Eq.~(\ref{avSDD2}) is expected to be valid for $np \gg 1$.
\end{itemize}

In Fig.~\ref{Fig1}(a) we plot $\left< SDD_\alpha(G) \right>$ as a function of the probability
$p$ of ER graphs of size $n=500$. All averages in Fig.~\ref{Fig1} are computed over
ensembles of $10^7/n$ random graphs. In Fig.~\ref{Fig1}(a) we show curves for $\alpha\in[0,4]$. 
The dashed-magenta curve corresponds to the case $\alpha=0$, which coincides with Eq.~(\ref{SDD0}).
Moreover, we observe that
$$\left<  SDD_{\alpha\le 0.5}(G) \right> \approx \left<  SDD_0(G) \right> = n(n-1) p \, .$$
However, once $\alpha> 0.5$, the curves $\left< SDD_\alpha(G) \right>$ versus $p$ deviate
from Eq.~(\ref{SDD0}), at intermediate values of $p$, in the form of a bump which is enhanced the
larger the value of $\alpha$ is.
Also, in Fig.~\ref{Fig1}(a) we can clearly see that Eq.~(\ref{avSDD}) is satisfied when $np\gg 1$,
as expected.

\begin{figure}[t!]
\begin{center}
\includegraphics[width=0.92\textwidth]{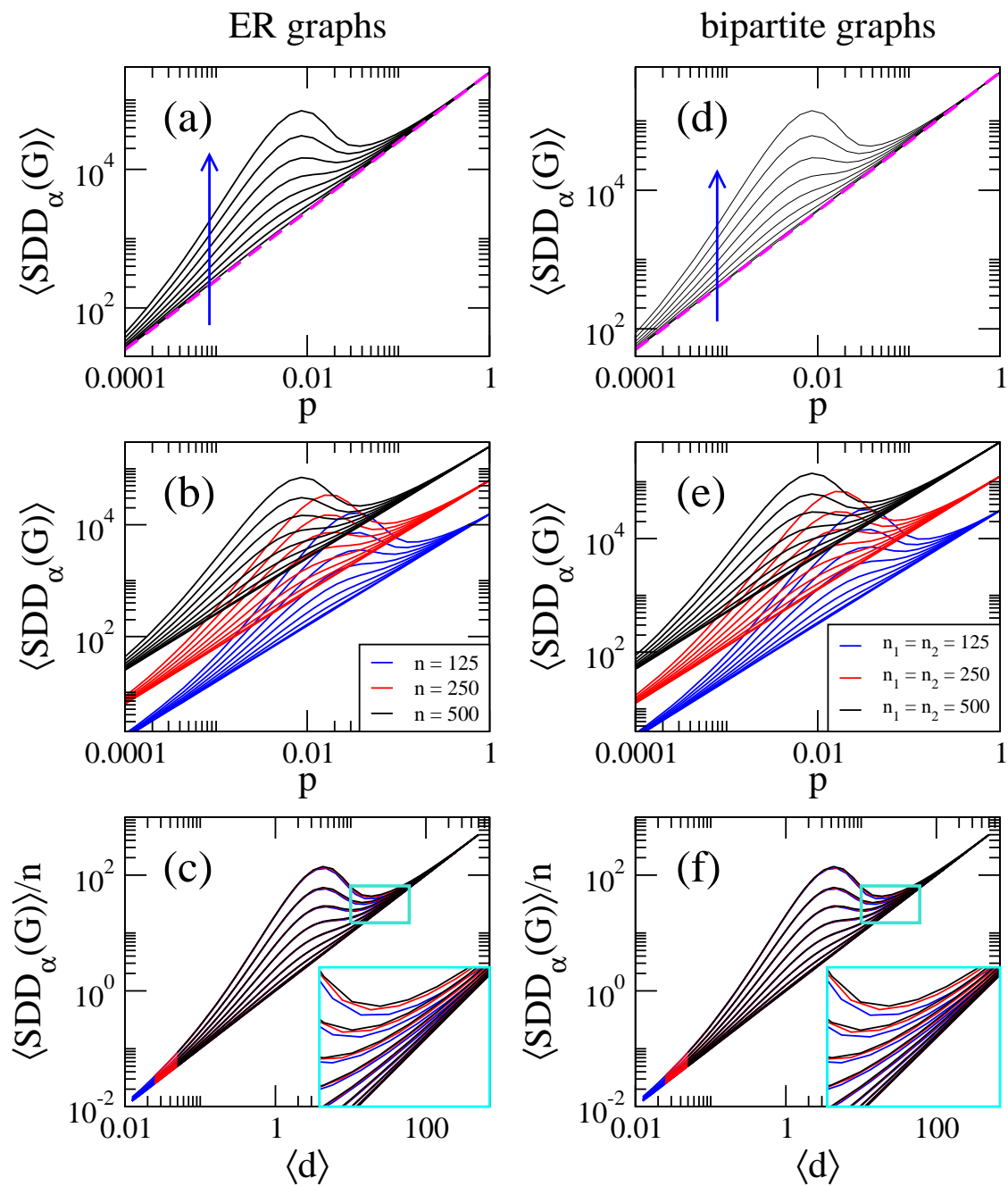}
\caption{
(a) Average variable symmetric division deg index $\left< SDD_\alpha(G) \right>$
as a function of the probability $p$ of Erd\"os-R\'enyi graphs of size $n=500$.
Here we show curves for $\alpha\in[0,4]$ in steps of $0.5$ (the arrow indicates increasing $\alpha$).
The dashed-magenta curve corresponds to the case $\alpha=0$.
(b) $\left< SDD_\alpha(G) \right>$ as a function of the probability $p$ of
ER graphs of three different sizes: $n=125$, 250 and 500.
(c) $\left< SDD_\alpha(G) \right>/n$ as a function of the average degree
$\left< d \right>$; same curves as in panel (b).
The inset in (c) is the enlargement of the cyan rectangle.
(d-f) Equivalent figures to (a-c), respectively, but for bipartite random graphs composed by sets of
equal sizes: in (d) $n_1=n_2=500$ while in (e,f) $n_1=n_2=\{125,250,500\}$.
}
\label{Fig1}
\end{center}
\end{figure}

Now, in Fig.~\ref{Fig1}(b) we present $\left< SDD_\alpha(G) \right>$ as a function of the probability
$p$ of ER graphs of three different sizes. It is clear from this figure that the blocks of curves,
characterized by the different graph sizes (and shown in different colors), display similar curves but
displaced on both axes. Moreover, the fact that these blocks of curves, plotted in semi-log
scale, are shifted the same amount on both $x-$ and $y-$axis when doubling $n$ make us anticipate
the scaling of $\left< SDD_\alpha(G) \right>$.
We stress that other average variable degree-based indices on ER random graphs (normalized to the 
graph size)
have been shown to scale with the average degree~\cite{AHMG20}. Indeed, this statement is encoded
in Eq.~(\ref{avSDD2}), that we derived for $np\gg 1$ but should serve as the global scaling
equation for $\left< SDD_\alpha(G) \right>$.

Therefore, in Fig.~\ref{Fig1}(c) we show $\left< SDD_\alpha(G) \right>/n$ as a function of the average
degree $\left< d \right>$ where the same curves of Fig.~\ref{Fig1}(b) have been used. There we verify
the global scaling of $\left< SDD_\alpha(G) \right>$, as anticipated in Eq.~(\ref{avSDD2}), by
noticing that the blocks of curves (painted in different colors) for different graph sizes fall on top of each
other.

Also, form Figs.~\ref{Fig1}(a-c) we observe that the inequality of Theorem~\ref{t:1} is 
extended to the average variable symmetric division deg index on random graphs:
\begin{equation}
\label{ineqt:1}
\left< SDD_{\a}(G) \right> \le \left< SDD_{\b}(G) \right> \, , \qquad 0 < \a < \b \, ;
\end{equation}
see e.g.~the blue arrow in Fig.~\ref{Fig1}(a) which indicates increasing $\alpha$. 
Here, the equality is attained if and only if $p=1$. However, we have observed that
$\left< SDD_{\a}(G) \right> \approx \left< SDD_{\b}(G) \right>$ already for $\left< d \right>\ge 10$.

\subsection{Average properties of the $SDD_\alpha$ index on bipartite random graphs}
\label{RR}

In Figs.~\ref{Fig1}(d,e) we present curves of the $\left< SDD_\alpha(G) \right>$ as a function of the
probability $p$ of BR graphs. For simplicity we show results for BR graphs composed by sets of equal 
sizes $n_1=n_2$. In Fig.~\ref{Fig1}(d) we consider the case of $n_1=n_2=500$ while in (e) we 
report $n_1=n_2=\{125,250,500\}$. In both figures we show curves for $\alpha\in[0,4]$ in steps of $0.5$.

Since edges in a bipartite graph join vertices of different sets, and we are labeling 
here the sets as set 1 and set 2, we replace $d_u$ by $d_1$ and $d_v$ by $d_2$ in the expression 
for the $SDD_\alpha(G)$ index below. Thus,
\begin{itemize}
\item[{\bf (i)}]
For $\alpha=0$, $\left< SDD_0(G)\right>$ gives twice the average number of edges of the BG
graph. That is,
\begin{eqnarray}
\label{SDD0BG}
\left<  SDD_0(G) \right> & = & \left<\sum_{E(G)} \left(\frac{d_1^0}{d_2^0} + \frac{d_2^0}{d_1^0} \right)\right>
= \left< \sum_{E(G)} (1+1) \right> \nonumber \\
& = & \left< 2 |E(G)| \right> = 2n_1n_2 p \, .
\end{eqnarray}
\item[{\bf (ii)}]
When both $n_1p\gg 1$ and $n_2p\gg 1$, we can approximate $d_1 \approx \left< d_1 \right>$ and 
$d_2 \approx \left< d_2 \right>$, then 
\begin{equation}
\left<  SDD_\alpha(G) \right> \approx
\left<\sum_{E(G)} \left(\frac{\left< d_1 \right>^\alpha}{\left< d_2 \right>^\alpha} + \frac{\left< d_2 \right>^\alpha}{\left< d_1 \right>^\alpha} \right)\right>
= \left<|E(G)| \left(\frac{\left< d_1 \right>^\alpha}{\left< d_2 \right>^\alpha} + \frac{\left< d_2 \right>^\alpha}{\left< d_1 \right>^\alpha} \right)\right> .
\label{avSDDBG}
\end{equation}
\item[{\bf (iii)}]
In the case we consider in Figs.~\ref{Fig1}(d-f), where $n_1=n_2=n/2$, so that 
$\left< d_1 \right>=\left< d_2 \right>=\left< d \right>$, Eq.~(\ref{avSDDBG}) reduces to
\begin{equation}
\left<  SDD_\alpha(G) \right> \approx \left< 2 |E(G)| \right> = 2n_1n_2 p = \frac{n^2}{2} p \, .
\label{avSDDBG2}
\end{equation}
\item[{\bf (iv)}]
By recognizing that $\left< d \right>=np/2$ we can rewrite Eq.~(\ref{avSDDBG2}) as
\begin{equation}
\frac{\left<  SDD_\alpha(G) \right>}{n} \approx \left< d \right> \, .
\label{avSDD2BG}
\end{equation}
We stress that Eq.~(\ref{avSDD2BG}) is expected to be valid for $np \gg 1$.
We also note that Eq.~(\ref{avSDD2BG}) has exactly the same form as Eq.~(\ref{avSDD2}).
\end{itemize}

From Figs.~\ref{Fig1}(d,e) we note that 
$$
\left<  SDD_{\alpha\le 0.5}(G) \right> \approx \left<  SDD_0(G) \right> = 2n_1n_2 p \, ,
$$
see the dashed-magenta curve in Fig.~\ref{Fig1}(d).
But once $\alpha> 0.5$, the curves $\left< SDD_\alpha(G) \right>$ versus $p$ deviate
from Eq.~(\ref{SDD0BG}), at intermediate values of $p$, in the form of bumps which are
enhanced the larger the value of $\alpha$ is. These bumps make clear the validity of
inequality (\ref{ineqt:1}) on BR graphs; see e.g.~the blue arrow in Fig.~\ref{Fig1}(d) 
which indicates increasing $\alpha$.

Finally, following the scaling analysis made in the previous subsection for ER graphs, in 
Fig.~\ref{Fig1}(f) we plot the $\left< SDD_\alpha(G) \right>/n$ as a function of the average
degree $\left< d \right>$ where the same data sets of Fig.~\ref{Fig1}(e) have been used. 
Thus we verify that $\left< SDD_\alpha(G) \right>$/n scales with $\left< d \right>$, as anticipated in 
Eq.~(\ref{avSDD2BG}); that is, the blocks of curves (painted in different colors) for different graph 
sizes coincide.

\section{Conclusions}

In this work we performed analytical and computational studies of the variable symmetric division deg index 
$SDD_\alpha(G)$.
First, we provided a monotonicity property and obtained new inequalities connecting $SDD_\alpha(G)$ with 
other well--known topological indices such as the first and second variable Zagreb indices, the
variable inverse sum deg index, as well as the the modified Narumi-Katayama index.
Then, we apply the index $SDD_\alpha(G)$ on two ensembles of random graphs: Erd\H{o}s-R\'enyi graphs and
bipartite random graphs. Thus, we computationally showed, for both random graph models, that the ratio 
$\bra SDD_\alpha(G) \ket/n$ is a function of the average degree $\bra d \ket$ only ($n$ being the order of 
the graph). We note that this last result, also observed for other variable topological indices~\cite{AHMG20}, 
is valid for random bipartite graphs only when they are formed by sets of the same size.

\section*{\bf ACKNOWLEDGEMENTS}
The research of J.M.R. and J.M.S. was supported by a grant from Agencia Estatal de Investigaci\'on (PID2019-106433GBI00/AEI/10.13039/501100011033), Spain.
J.M.R. was supported by the Madrid Government (Comunidad de Madrid-Spain) under the Multiannual Agreement with UC3M in the line of Excellence of University Professors (EPUC3M23), and in the context of the V PRICIT (Regional Programme of Research and Technological Innovation).


\end{document}